\documentclass[reqno,11pt]{amsart}
\usepackage{amsfonts,bm}
\usepackage{cite}
\usepackage{url}
\usepackage{hyperref,comment}

\newtheorem{theorem}{Theorem}[section]
\newtheorem{lemma}[theorem]{Lemma}
\newtheorem{proposition}[theorem]{Proposition}
\newtheorem{corollary}[theorem]{Corollary}

\newcommand{\re}{\mathrm{Re}}
\newcommand{\norm}[1]{\left\lVert #1 \right\rVert}  
 \newcommand{\inner}[1]{\left\langle #1 \right\rangle} 
\newcommand{\tr}{\mathrm{tr}}

\author{Yang Huang}
\address{Yang Huang, College of Mathematics and Econometrics, Hunan University, 
Changsha, Hunan, 410082, P.R. China}
\email{\url{FairyHuang@csu.edu.cn} }

\author{Yongtao  Li}
\address{ Yongtao Li, College of Mathematics and Econometrics, Hunan University, 
Changsha, Hunan, 410082, P.R. China}
\email{\url{ytli0921@hnu.edu.cn} }

\author{Lihua Feng}
\address{Lihua Feng, School of Mathematics and Statistics, Central South University, 
Changsha, Hunan, 410083, P.R. China}
\email{\url{fenglh@163.com} }

\author{Weijun Liu$^*$}
\address{Weijun Liu, School of Mathematics and Statistics, Central South University, 
Changsha, Hunan, 410083, P.R. China}
\email{\url{wjliu6210@126.com}}

\keywords{Block matrices; Positive semidefinite; Generalized matrix function; Krein's inequality  }
\subjclass[2010]{47B65, 15B42, 15A45}
\date{March 2, 2022. 
This paper was firstly finished in March, 2020, and was later published 
on Journal of Mathematical Inequalities, 16 (2022), no.1, 143--156. 
See \url{http://jmi.ele-math.com/16-11}. }

\begin{document}

\title[Generalized Matrix Function and inner product]{Inequalities for generalized matrix function 
and inner product}

\begin{abstract}
We present inequalities related to generalized matrix function for positive semidefinite 
block matrices. We introduce partial generalized matrix functions 
corresponding to partial traces, and 
then provide a unified extension of the recent inequalities  due to 
Lin [Electron. J. Linear Algebra 27 (2014) 821--826], Zhang et al. [Linear Algebra Appl. 498 (2016) 99--105] 
and [Electron. J. Linear Algebra 27 (2014) 332--341] 
and a result of Choi [Linear Algebra Appl. 532 (2017) 1--7]. 
Moreover, we demonstrate the application of a positive semidefinite $3\times 3$ block  matrix, 
which motivates us to give  alternative proofs of Dragomir's inequality and 
Krein's inequality. 
\end{abstract}
\maketitle

\section{Introduction}

Let $G$ be a subgraph of the symmetric group $S_n$ on $n$ letters 
and let $\chi$ be an irreducible character of $G$. 
For any $n\times n$ complex matrix $A=(a_{ij})$, 
the generalized matrix function of $A$ 
(also known as immanant) afforded by $G$ and $\chi$ is defined by 
\[ \mathrm{d}_{\chi}^G(A):= \sum\limits_{\sigma \in G} 
\chi (\sigma) \prod\limits_{i=1}^n a_{i\sigma (i)}. \]

Some specific subgroups $G$ and characters $\chi$ lead to some well-known functionals 
on the matrix space. For example, 
if $G=S_n$ and $\chi$ is the signum function with value $\pm 1$, 
then the generalized matrix function becomes the usual matrix determinant. Moreover, 
by setting $\chi (\sigma) \equiv 1$ for each $\sigma \in G=S_n$, we get the permanent of the matrix. In addition,  
setting $G=\{e\}\subset S_n$ defines the product of the main diagonal entries of the matrix 
(also known as the Hadamard matrix function).

Let $A$ and $B$ be $n\times n$ positive semidefinite matrices. 
It is easy to prove by simultaneous diagonalization argument  that 
\begin{equation} \label{eeq0}
\det (A+B) \ge \det (A) +\det (B). 
\end{equation}
There are many extensions and generalizations of (\ref{eeq0}) in the literature; see, e.g., \cite{Lin14,Zhang14,Zhang16,BS15}.  
For instance, 
the following remarkable extension  is known (see \cite[p. 228]{Me97}), it states that 
\begin{equation} \label{eeq1}
\mathrm{d}_{\chi}^G (A+B) \ge \mathrm{d}_{\chi}^G(A) 
+ \mathrm{d}_{\chi}^G(B). 
\end{equation}
 Recently, Paksoy, Turkmen and Zhang \cite{Zhang14} provided a natural generalization 
of (\ref{eeq1}) for triple matrices 
by embedding the vectors of Gram matrices into a 
``sufficiently large'' inner product space and  
using the properties of tensor products. 
More precisely,  they proved for positive semidefinite matrices $A,B$ and $C$ that 
\begin{equation} \label{eeq2}
\mathrm{d}_{\chi}^G(A+B+C) +
\mathrm{d}_{\chi}^G(C) \ge 
\mathrm{d}_{\chi}^G(A+C) +\mathrm{d}_{\chi}^G(B+C). 
\end{equation}
Their approach to establish (\ref{eeq2}) is algebraic as well as combinatorial. 
Soon after, 
Chang, Paksoy and Zhang \cite[Theorem 3]{Zhang16} 
(Berndt and Sra \cite{BS15} independently) 
presented a further improvement of (\ref{eeq2}) 
by considering the tensor products of operators as words on certain alphabets, 
which states that 
\begin{equation} \label{eeq3} \begin{aligned} 
 &\mathrm{d}_{\chi}^G (A+B+C) +
 \mathrm{d}_{\chi}^G (A) + 
 \mathrm{d}_{\chi}^G (B) +\mathrm{d}_{\chi}^G (C) \\
 & \quad \ge \mathrm{d}_{\chi}^G (A+B) + 
 \mathrm{d}_{\chi}^G (A+C)+ \mathrm{d}_{\chi}^G (B+C). 
\end{aligned}  \end{equation}
Indeed, (\ref{eeq3}) is an improvement on (\ref{eeq2})  because 
\begin{align*}
&\mathrm{d}_{\chi}^G (A+B+C)  +\mathrm{d}_{\chi}^G (C) 
 - \bigl( \mathrm{d}_{\chi}^G (A+C) + 
 \mathrm{d}_{\chi}^G (B+C) \bigr) \\
& \quad \ge \mathrm{d}_{\chi}^G (A+B)- 
\mathrm{d}_{\chi}^G (A) -\mathrm{d}_{\chi}^G (B) \ge 0. 
\end{align*}

Before starting our results, 
we  review briefly the basic definition and 
notation of tensor product in Multilinear Algebra Theory \cite{Me97}. 
The space of $m\times n$ complex matrices is denoted by $\mathbb{M}_{m\times n}$. 
If $m=n$, we use $\mathbb{M}_n$ instead of $\mathbb{M}_{n\times n}$ and 
if $n=1$, we use $\mathbb{C}^m$ instead of $\mathbb{M}_{m\times 1}$. 
The  identity matrix of $\mathbb{M}_n$ is denoted by $I_n$, or simply by $I$ 
if no confusion is possible. 
We use $\mathbb{M}_m(\mathbb{M}_n)$ for the set of $m\times m$ block matrices 
with each block in $\mathbb{M}_n$. 
Let $X\otimes Y$ denote the Kronecker product of $X,Y$, that is, 
if $X=[x_{ij}]\in \mathbb{M}_m$ and $Y\in \mathbb{M}_n$, then 
$X\otimes Y\in \mathbb{M}_m(\mathbb{M}_n)$ whose $(i,j)$ block is $x_{ij}Y$. 
By convention, if $X\in \mathbb{M}_n$ is positive semidefinite, we write $X\ge 0$. 
For two Hermitian matrices $A$ and $B$ of the same size, $A\ge B$ means $A-B\ge 0$.  

Let $V$ be an $n$-dimensional Hilbert space 
and $\otimes^n V$ be the tensor product space of $n$ copies of $V$. 
Let $G$ be a subgroup of the symmetric group $S_n$ and 
$\chi$ be an irreducible character of $G$. 
The {\it symmetrizer} induced by $\chi$ on the tensor product space $\otimes^n V$ 
is defined by its action 
\begin{equation} \label{eeq5}
 S(v_1\otimes \cdots \otimes v_n):=\frac{1}{|G|} \sum\limits_{\sigma \in G}
\chi (\sigma) v_{\sigma^{-1}(1)}\otimes \cdots \otimes v_{\sigma^{-1}(n)}.  
\end{equation}
All elements of the form (\ref{eeq5}) span a vector space, denoted by 
$V_{\chi}^n(G)\subset \otimes^n V$, which is the space of the symmetry class of tensors 
associated with $G$ and $\chi$ (see \cite[p. 154, 235]{Me97}). 
It is easy to verified that $V^n_{\chi}(G)$ is an invariant subspace of $\otimes^n V$. 
For a linear operator $A$ on $V$, the induced operator $K(A)$ of $A$ with respect to $G$ and $\chi$ 
is defined to be $K(A)=(\otimes^n A)\big|_{V^n_{\chi}(G)}$, 
the restriction of $\otimes^n A$ on $V_{\chi}^n(G)$. 

The induced operator $K(A)$ is closely related to generalized matrix function. 
Let $e_1,e_2,\ldots ,e_n$ be an orthonormal basis of $V$ and 
$P$  be a matrix representation of  
the linear operator $A$ on $V$ with respect to  the basis $e_1,\ldots ,e_n$. Then 
\begin{equation} \label{eeq6}
 \mathrm{d}_{\chi}^G \left( P^T\right)=\frac{|G|}{\mathrm{deg}(\chi)} \langle K(A)e^*, e^*\rangle, 
\end{equation}
where $\mathrm{deg}(\chi)$ is the degree of $\chi$ and $e^*:=e_1*e_2* \cdots *e_n$ 
is the decomposable symmetrized tensor of $e_1,\ldots ,e_n$ (see \cite[p. 227, 155]{Me97}). 

In this paper, we first extend the above-cited results (\ref{eeq1}), 
(\ref{eeq2}) and (\ref{eeq3}) to positive semidefinte block matrices in Section \ref{sec2}, 
and then we obtain some inequalities related to determinant and permanent as  byproducts. 
In Section \ref{sec3}, we investigate the applications 
of a positive semidefinite $3\times 3$ block matrix and 
provide an alternative proof of Dragomir's 
inequality. In Section \ref{sec4}, 
we present a simple proof of Krein's inequality and other triangle inequalities.

\section{Partial Generalized Matrix Functions}
\label{sec2}

Now we introduce the definition of partial traces.  
Partial trace becomes a popular topic recently and it
has various applications in
 Quantum Information Theory \cite[p. 12]{Petz}. 
Given $A\in \mathbb{M}_m(\mathbb{M}_n)$, 
the first partial trace (map) $A \mapsto \mathrm{tr}_1 A \in \mathbb{M}_n$ is defined as the  
adjoint map of the imbedding map $X \mapsto I_m\otimes X\in \mathbb{M}_m\otimes \mathbb{M}_n$. 
Correspondingly, the second partial trace (map)  $A \mapsto \mathrm{tr}_2 A\in \mathbb{M}_m$ is 
defined as the adjoint map of the imbedding map 
$Y\mapsto Y\otimes I_n \in \mathbb{M}_m\otimes \mathbb{M}_n$. Therefore, we have
\[ \langle I_m\otimes X, A \rangle =\langle X, \mathrm{tr}_1A \rangle ,
\quad \forall X\in \mathbb{M}_n, \]
and 
\[ \langle Y\otimes I_n, A \rangle =\langle Y,\mathrm{tr}_2 A \rangle, 
\quad \forall Y\in \mathbb{M}_m. \]
Assume that $A=[A_{ij}]_{i,j=1}^m$ with $A_{ij}\in \mathbb{M}_n$, 
then the visualized forms of the partial traces 
are actually given in  \cite[Proposition 4.3.10]{Bh07} as
\[ \mathrm{tr}_1 { A}=\sum\limits_{i=1}^m A_{ii},\quad 
\mathrm{tr}_2{ A}=[\mathrm{tr}A_{ij}]_{i,j=1}^m. \] 
Motivated by the definition of partial traces, 
we now define the {\it partial generalized matrix function} by 
\begin{equation*}
 \left( \mathrm{d}_{\chi}^G\right)_{\! 2} \! A = \left[ 
 \mathrm{d}_{\chi}^G (A_{ij}) \right]_{i,j=1}^m \in \mathbb{M}_m. 
\end{equation*}
Suppose that $A_{ij}=\bigl[a_{rs}^{ij} \bigr]_{r,s=1}^n$, 
 we next define $B_{rs}:=\bigl[ a_{rs}^{ij}\bigr]_{i,j=1}^m$ and 
\[ \left( \mathrm{d}_{\chi}^G\right)_{\! 1} \! A 
= \left[ \mathrm{d}_{\chi}^G (B_{rs}) \right]_{r,s=1}^n 
\in \mathbb{M}_n. \]
 Furthermore, 
we denote by $\widetilde{A}=[B_{rs}]_{r,s=1}^n$,  
then we can easily see that  
$( \mathrm{d}_{\chi}^G)_1 A =( \mathrm{d}_{\chi}^G)_2 
\widetilde{A}$ and $\widetilde{\widetilde{A}}=A$.   
Additionally,  $\widetilde{A}$ and $A$ are unitarily similar; 
see, e.g., \cite[Theorem 7]{Choi17}. 
We remark here that the motivation of the 
definition of partial  generalized matrix 
function also grew out of that of 
the partial determinants \cite{Choi17}.

Let $A=[A_{ij}]_{i,j=1}^m\in \mathbb{M}_m(\mathbb{M}_n)$ 
be a positive semidefinite block matrix. 
It is well-known that $\tr_1 A 
\in \mathbb{M}_m$ is  positive semidefinite. 
Moreover, 
both $\mathrm{tr}_2A=[\mathrm{tr} A_{ij}]_{i,j=1}^m$ and 
$\det_2 A=[\det A_{ij}]_{i,j=1}^m$ 
 are also positive semidefinite;  
see \cite[p. 221, 237]{Zhang11}. 
Whereafter, Zhang \cite[Theorem 3.1]{Zhang12} extends the positivity to generalized matrix function via the 
generalized Cauchy-Binet formula, 
i.e., the matrix 
$\left[\mathrm{d}_{\chi}^G (A_{ij}) \right]_{i,j=1}^m$ is positive semidefinite. 

It is interesting that Lin and Sra \cite{LS16} proved 
for $A,B\ge 0$, 
\begin{equation}\label{eqls16}
 {\det}_2 (A+B)\ge {\det}_2 A +{\det}_2 B. 
 \end{equation}
 In \cite[Corollary 9]{Choi17}, an analogous result 
 corresponding to the first partial determinant 
  is also proved, it states that 
  \begin{equation}\label{eqchoi17}
 {\det}_1 (A+B)\ge {\det}_1 A +{\det}_1 B. 
 \end{equation}

The following lemma plays an important role in the proof of 
our extension (Theorem \ref{thm23}),  
and it also can be found as a special case 
 in \cite{Zhang16} and \cite{BS15}. 
We here provide a proof for the sake of completeness.

\begin{lemma} \label{lem32}
Let $A,B,C$ be positive semidefinite matrices of same size. 
Then for every positive integer $r$, we have
\begin{equation}  \label{eq13}
\begin{aligned}
& \otimes^r (A+B+C) +\otimes^r A +\otimes^r B +\otimes^r C \\ 
&\quad \ge \otimes^r (A+B) +\otimes^r (A+C) +\otimes^r (B+C).
\end{aligned} \end{equation}
\end{lemma} 

\begin{proof}
The proof is by induction on $r$. 
The base case $r=1$  holds with equality, 
and the case $r=2$ is easy to verify. 
Assume therefore (\ref{eq13}) holds for some $r=m\ge 2$, that is 
\begin{equation*}
\begin{aligned}
& \otimes^m (A+B+C) +\otimes^m A +\otimes^m B +\otimes^m C \\ 
&\quad \ge \otimes^m (A+B) +\otimes^m (A+C) +\otimes^m (B+C).
\end{aligned} \end{equation*}
For $r=m+1$, we have  
\begin{align*}
&\otimes^{m+1}(A+B+C)\\
&= \bigl( \otimes^m(A+B+C) \bigr) \otimes (A+B+C)\\
& \ge \bigl( \otimes^m (A+B) +\otimes^m (A+C) +\otimes^m (B+C) -
\otimes^m A - \otimes^m B - \otimes^m C \bigr) \\
  &\quad    \otimes (A+B+C) \\
& = \otimes^{m+1}(A+B) +\otimes^{m+1}(A+C) +\otimes^{m+1}(B+C) \\
& \quad     -\otimes^{m+1}A-\otimes^{m+1}B-\otimes^{m+1}C \\
& \quad + \bigl( \otimes^m(A+B) \bigr) \otimes C + 
\bigl(\otimes^m(A+C)\bigr) \otimes B  + \bigl(\otimes^m(B+C)\bigr) \otimes A \\
& \quad - \bigl(\otimes^m A \bigr) \otimes (B+C) -\bigl(\otimes^m B \bigr) \otimes (A+C) 
- \bigl(\otimes^m C \bigr) \otimes (A+B).
\end{align*}
It remains to prove the following result   
\begin{align*}
&\bigl( \otimes^m(A+B) \bigr) \otimes C + 
\bigl(\otimes^m(A+C)\bigr) \otimes B  + \bigl(\otimes^m(B+C)\bigr) \otimes A \\
&\quad \ge  \bigl(\otimes^m A \bigr) \otimes (B+C)  + \bigl(\otimes^m B \bigr) \otimes (A+C) 
  + \bigl(\otimes^m C \bigr) \otimes (A+B).
\end{align*}
This follows immediately by the superadditivity of 
tensor power: 
\begin{align*} 
\otimes^m (A+B)&\ge \otimes^m A +\otimes^m B, \\
\otimes^m(A+C) &\ge \otimes^m A +\otimes^m C, \\
\otimes^m(B+C) &\ge \otimes^m B +\otimes^m C. 
\end{align*}
Thus, the desired inequality  holds.
\end{proof}

To show our main result in this section, we require one more lemma. 

\begin{lemma}  \label{lem10} 
(\cite[p. 93]{Bh07})
Let $A=[A_{ij}]_{i,j=1}^m \in \mathbb{M}_m(\mathbb{M}_n)$. 
Then $[\otimes^r A_{ij}]_{i,j=1}^m $ is a principal submatrix of $\otimes^r A$ 
for every positive integer $r$.
\end{lemma}

We now present a unified extension  of 
(\ref{eeq3}), (\ref{eqls16}) and (\ref{eqchoi17}). 

\begin{theorem} \label{thm23}
Let ${A},B,C\in \mathbb{M}_m(\mathbb{M}_n)$ be positive semidefinite. Then
\begin{equation} \label{eq17} \begin{aligned} 
 &\left( \mathrm{d}_{\chi}^G\right)_{1}(A+B+C) +
\left( \mathrm{d}_{\chi}^G\right)_1 A + 
\left( \mathrm{d}_{\chi}^G\right)_1B + 
\left( \mathrm{d}_{\chi}^G\right)_1 C \\
 & \quad \ge 
  \left( \mathrm{d}_{\chi}^G\right)_1 (A+B) +
   \left( \mathrm{d}_{\chi}^G\right)_1 (A+C)
+\left( \mathrm{d}_{\chi}^G\right)_1(B+C), 
\end{aligned}  \end{equation}
and 
\begin{equation} \label{eq18} \begin{aligned} 
 &\left( \mathrm{d}_{\chi}^G\right)_2(A+B+C) +
\left( \mathrm{d}_{\chi}^G\right)_2 A + 
\left( \mathrm{d}_{\chi}^G\right)_2 B + 
\left( \mathrm{d}_{\chi}^G\right)_2 C \\
 & \quad \ge 
 \left( \mathrm{d}_{\chi}^G\right)_2 (A+B) +
  \left( \mathrm{d}_{\chi}^G\right)_2 (A+C)
+\left( \mathrm{d}_{\chi}^G\right)_2 (B+C). 
\end{aligned}  \end{equation}
\end{theorem}

\begin{proof}
We only prove (\ref{eq18}) 
since (\ref{eq17}) can be proved by exchanging the role of $\widetilde{A}$ and $A$. 
By Lemma \ref{lem32}, we have 
\begin{equation*} 
\begin{aligned}
& \otimes^r (A+B+C) +\otimes^r A +\otimes^r B +\otimes^r C \\ 
&\quad \ge \otimes^r (A+B) +\otimes^r (A+C) +\otimes^r (B+C).
\end{aligned} \end{equation*}
By Lemma \ref{lem10}, it follows that  
\begin{align*} 
 & [\otimes^r (A_{ij}+B_{ij}+C_{ij})]_{i,j=1}^m 
+ [\otimes^r A_{ij}]_{i,j=1}^m + [\otimes^r B_{ij}]_{i,j=1}^m +[\otimes^r C_{ij}]_{i,j=1}^m \\
 &\quad \ge [\otimes^r (A_{ij}+B_{ij})]_{i,j=1}^m + 
[\otimes^r (A_{ij}+C_{ij})]_{i,j=1}^m + [\otimes^r (B_{ij}+C_{ij})]_{i,j=1}^m. 
\end{align*}
By restricting above inequality to the symmetry class $V_{\chi}^G(V)$, we get 
\begin{align*}
& [ K(A_{ij}+B_{ij}+C_{ij})]_{i,j=1}^m 
+ [K (A_{ij})]_{i,j=1}^m + [K (B_{ij})]_{i,j=1}^m +[K (C_{ij})]_{i,j=1}^m \\
&\quad \ge [K (A_{ij}+B_{ij})]_{i,j=1}^m +
[K (A_{ij}+C_{ij})]_{i,j=1}^m + [K (B_{ij}+C_{ij})]_{i,j=1}^m. 
\end{align*}
The required result (\ref{eq18}) follows by combining (\ref{eeq6}). 
\end{proof}

\begin{corollary} \label{coro24}
Let $A,B,C\in \mathbb{M}_m(\mathbb{M}_n)$ be positive semidefinite. Then
\begin{equation*}  \begin{aligned} 
 &\mathrm{det}_1(A+B+C) +\mathrm{det}_1A +\mathrm{det}_1B +\mathrm{det}_1 C \\
 & \quad \ge \mathrm{det}_1(A+B) + \mathrm{det}_1 (A+C)+\mathrm{det}_1(B+C), 
\end{aligned}  \end{equation*}
and 
\begin{equation*}  \begin{aligned} 
& \mathrm{det}_2(A+B+C) +\mathrm{det}_2A +\mathrm{det}_2B +\mathrm{det}_2 C \\
& \quad  \ge \mathrm{det}_2(A+B) + \mathrm{det}_2 (A+C)+\mathrm{det}_2(B+C).
\end{aligned}  \end{equation*}
\end{corollary}

\begin{proof}
We set $G=S_n$, the symmetric group and  
specify $\chi =\mathrm{sign}$, the signum function, then 
 Theorem \ref{thm23} yields the desired inequalities. 
\end{proof}

\noindent
{\bf Remark.}~~
Corollary \ref{coro24}
is a unified extension of a large number of 
 determinantal inequalities. 
In particular, by setting $m=1$, we can get 
\begin{align*} &\det (A+B+C) +\det A+ \det B+\det C \\
&\quad \ge \det (A+B) +\det (A+C) +\det (B+C).  
\end{align*}
Therefore, we can obtain  
\[ \det (A+B+C) +\det C \ge \det (A+C) +\det (B+C). \]
These two inequalities are well-known, 
the first one is the main result in \cite{Lin14} and 
the second  can be found in \cite[p. 215, Problem 36]{Zhang11}. 
On the other hand, 
Corollary \ref{coro24} also implies the following result  
\begin{align*}
 {\det}_2 (A+B+C) +{\det}_2(C)\ge {\det}_2 (A+C) +{\det}_2 (B+C),
\end{align*}
and 
\[ {\det}_2 (A+B)\ge {\det}_2 A +{\det}_2 B, \]
which are generalizations of a recent result in \cite{LS16}.

Similarly, we can get some analogous inequalities for permanent.

\begin{corollary} \label{coro25}
Let $A,B,C\in \mathbb{M}_m(\mathbb{M}_n)$ be positive semidefinite. Then
\begin{equation*}  \begin{aligned} 
 &\mathrm{per}_1(A+B+C) +\mathrm{per}_1A +\mathrm{per}_1B +\mathrm{per}_1 C \\
 & \quad \ge \mathrm{per}_1(A+B) + \mathrm{per}_1 (A+C)+\mathrm{per}_1(B+C), 
\end{aligned}  \end{equation*}
and 
\begin{equation*}  \begin{aligned} 
& \mathrm{per}_2(A+B+C) +\mathrm{per}_2A +\mathrm{det}_2B +\mathrm{per}_2 C \\
& \quad  \ge \mathrm{per}_2(A+B) + \mathrm{per}_2 (A+C)+\mathrm{per}_2(B+C).
\end{aligned}  \end{equation*}
\end{corollary}

\begin{proof}
We set $G=S_n$, the symmetric group and  
specify $\chi (g)\equiv 1$ for all $g\in S_n$, then 
applying Theorem \ref{thm23} leads to the required inequalities. 
\end{proof}

\section{Positivity and Dragomir's inequality}
\label{sec3}

Recently, positive semidefinite $3\times 3$ block matrices are extensively studied, 
such a partition leads to versatile and 
elegant theoretical inequalities; 
see, e.g., \cite{Lin14b,GR15,Drury14} for related results. 
In particular, assume that $X,Y,Z$ are matrices with appropriate size, then the block matrix
\begin{equation} \label{eeq9} \begin{bmatrix}
X^*X & X^*Y & X^*Z \\
Y^*X & Y^*Y & Y^*Z \\
Z^*X & Z^*Y & Z^*Z
\end{bmatrix} 
\end{equation}
is positive semidefinite. As stated in Section \ref{sec2}, it follows that  
the resulting matrices by taking determinant and trace entrywise  
are positive semidefinite. 
Different size of matrices in (\ref{eeq9}) 
will yield a large number of interesting triangle inequalities; see \cite[Theorem 2.1]{Lin14b}.  
Particularly, 
let $X,Y,Z$ be column vectors in $\mathbb{C}^n$, 
say $u,v,w$. We write $\inner{\cdot , \cdot }$ 
 for the standard inner product on $\mathbb{C}^n$, 
 then it follows  that 
\begin{equation}\label{eeq12} 
\begin{bmatrix}
\re \inner{u,u} & \re \inner{u,v} & \re \inner{u,w} \\[0.1cm]
\re \inner{v,u}& \re \inner{v,v} & \re \inner{v,w} \\[0.1cm]
\re \inner{w,u}  & \re \inner{w,v} & \re \inner{w,w}
\end{bmatrix} 
\end{equation}
is a positive semidefinite matrix.  
This matrix was widely studied in the literature;  
see \cite{Lin12,Zhang16a} for more details.

In this section, 
we first present two analogous result of (\ref{eeq12}) 
in Corollary \ref{coro32} and Proposition \ref{prop33}. 
By applying Corollary \ref{coro32},  we then 
give a short proof of Dragomir's inequality (Theorem \ref{thm32}). 
Next, we require the following lemma, 
which is an exercise in \cite[p. 26]{Bh07}. 
We here provide a detailed  proof 
  for the convenience of readers.

\begin{lemma} \label{lem1}
Let $A=[a_{ij}]$ be a $3\times 3$ complex matrix and 
let $|A|=[|a_{ij}|]$ 
be the matrix obtained from $A$ by taking the absolute values of the entries of $A$. 
If $A$ is positive semidefinite, then $|A|$ is positive semidefinite. 
\end{lemma}

\begin{proof}
We first note that the positivity of $A$ implies  all diagonal entries of $A$ are nonnegative. 
If a diagonal entry of $A$ is zero, as $A$ is positive semidefinite, 
then the entire row entries and column entries of $A$ 
are zero and   
it is obvious that the positivity of $\begin{bmatrix}\begin{smallmatrix}a & c \\ 
\overline{c}&b \end{smallmatrix}\end{bmatrix}$ implies the positivity of 
$\begin{bmatrix} \!\begin{smallmatrix}|a| & |c| \\ 
|\overline{c}| & |b| \end{smallmatrix}\!\end{bmatrix}$. 
Without loss of generality, we may assume that $a_{ii}>0$ for $i=1,2,3$. 
Let $D=\mathrm{diag}\bigl\{ a_{11}^{-1/2},a_{22}^{-1/2},a_{33}^{-1/2}\bigr\}$ 
and observe that $D^*|A|D=|D^*AD|$. By scaling, we further assume that 
\[ A=\begin{bmatrix} 1 &a &b \\ \overline{a} &1 &c \\ \overline{b} &\overline{c} &1
\end{bmatrix}. \]
Recall that $X\ge 0$ means $X$ is positive semidefinite. 
Our goal is to prove 
\begin{equation} \label{eeq14}
 \begin{bmatrix} 1 &a &b \\ \overline{a} &1 &c \\ \overline{b} &\overline{c} &1
\end{bmatrix}
\ge 0 
\Rightarrow 
\begin{bmatrix} 1 &|a| &|b| \\ |\overline{a}| &1 &|c| \\ |\overline{b}| &|\overline{c}| &1
\end{bmatrix}\ge 0 . 
\end{equation}
Assume that $a=|a|e^{i\alpha}$ and $b=|b|e^{i\beta}$, 
and denote $Q=\mathrm{diag}\left\{ 1,e^{-i\alpha },e^{-i\beta }\right\}$. 
By a direct computation,  we obtain 
\[ Q^*AQ =\begin{bmatrix}
1 & |a| & |b| \\ |a| & 1 & ce^{i(\alpha -\beta)} \\ 
|b| &\overline{c}e^{i(\beta -\alpha)} &1 
\end{bmatrix}. \]
Since $A\ge 0$, taking the determinant of $Q^*AQ$ gives 
\[ 1+|a||b|\left( ce^{i(\alpha -\beta)} +\overline{c}e^{i(\beta -\alpha)}\right) 
\ge |a|^2 +|b|^2+|c|^2. \]
Note that $ 2|c| \ge 2\,\mathrm{Re} \left( ce^{i(\alpha -\beta)} \right) \ge 
\left( ce^{i(\alpha -\beta)} +\overline{c}e^{i(\beta -\alpha)}\right)$, then 
\[ 1+2|a||b||c| 
\ge |a|^2 +|b|^2+|c|^2, \]
which is actually  $\det |A|\ge 0$. 
Combining $1-|a|^2\ge 0$, that is, 
every principal minor of $|A|$ is nonnegative, then  
\[ \begin{bmatrix} 1 &|a| &|b| \\ 
|\overline{a}| &1 &|c| \\ 
|\overline{b}| &|\overline{c}| &1
 \end{bmatrix}\ge 0 . \]
Thus, the desired statement (\ref{eeq14}) now follows. 
\end{proof}

\noindent 
{\bf Remark.}~~
We remark that the converse of Lemma \ref{lem1} is not true 
and the statement also not hold for $4\times 4$ case. For example, setting 
\[
B=\begin{bmatrix}
1&-1&-1 \\ -1& 1&-1 \\ -1 &-1 &1 
\end{bmatrix},\quad 
C=\begin{bmatrix}
10 & 3 &-2 &1 \\ 3& 10 &0 &9 \\ -2&0&10 &4 \\ 1&9&4&10
\end{bmatrix}. \] 
We can easily check that both $|B|$ and $C$ are positive semidefinite.  
 However, $B$ and $|C|$ are not positive semidefinite  
because $\det B=-4$ and $\det |C|=-364$. 

\begin{corollary} \label{coro32}
If $u,v$ and $w$ are vectors in an inner product space, then 
\begin{equation*} \begin{bmatrix}
\bigl| \inner{u,u}\bigr| & \bigl| \inner{u,v}\bigr| & \bigl|\inner{u,w}\bigr| \\[0.15cm]
\bigl|\inner{v,u}\bigr| & \bigl|\inner{v,v}\bigr| & \bigl|\inner{v,w}\bigr| \\[0.15cm]
\bigl|\inner{w,u}\bigr|  & \bigl|\inner{w,v}\bigr| & \bigl|\inner{w,w}\bigr|
\end{bmatrix} 
\end{equation*}
is a positive semidefinite matrix. 
\end{corollary}

\begin{proof}
By  the positivity of Gram matrix and Lemma \ref{lem1}. 
\end{proof}

\begin{proposition} \label{prop33}
If $u,v$ and $w$ are  vectors in a Euclidean space such that $u+w=v$,  
then the following matrix is positive semidefinite. 
\begin{equation*} \begin{bmatrix}
 \inner{u,u} &  \inner{u,v} &  -\inner{u,w} \\[0.1cm]
 \inner{v,u}&  \inner{v,v} &  \inner{v,w} \\[0.1cm]
 -\inner{w,u}  &  \inner{w,v} &  \inner{w,w}
\end{bmatrix} 
\end{equation*}
\end{proposition}

\begin{proof}
We  choose an orthonormal basis of $\mathrm{Span}\{u,v,w\}$, then  
we may assume that $u,v$ and $w$ are vectors in $\mathbb{R}^3$ and form 
a triangle on a plane. We denote the angle of $u,v$ by $\alpha$, angle of $-u,w$ by $\beta$ 
and angle of $-w,-v$ by $\gamma$, respectively. Note that $\alpha +\beta +\gamma =\pi$, 
 we next will prove that 
\[ \cos^2 \alpha + \cos^2 \beta +\cos^2 \gamma +2 \cos \alpha \cos \beta \cos \gamma =1. \]
Invoking the fact 
$ \cos^2 x =\frac{1+\cos (2x)}{2}$ and $
\cos x + \cos y = 2\cos \frac{x+y}{2} \cos \frac{x-y}{2}$, 
and combining with $\cos (\alpha + \beta) =
\cos (\pi - \gamma)=\cos \gamma$, we can obtain 
\begin{align*}
\cos^2 \alpha +\cos^2 \beta +\cos^2 \gamma  
&=1+\frac{1}{2}(\cos 2\alpha +\cos 2\beta) +\cos^2 \gamma \\
& =1+ \cos (\alpha +\beta) 
\cos (\alpha- \beta) +\cos^2 \gamma \\
&=1-\cos \gamma \bigl(  \cos (\alpha -\beta) + 
\cos (\alpha +\beta) \bigr) \\
& =1-\cos \gamma \cdot 2\cos \alpha \cos \beta,
\end{align*}
By computing the principal minor, it follows that 
\[ R:=\begin{bmatrix}
1  & \cos \alpha & \cos \beta \\
\cos \alpha &1 & \cos \gamma \\
\cos \beta & \cos \gamma & 1
\end{bmatrix} \]
is positive semidefinite. Setting $S=\mathrm{diag}\{\norm{u},\norm{v},\norm{w}\}$. 
Thus  $S^TRS$ is positive semidefinite. This completes the proof. 
\end{proof}

In 1985, 
Dragomir \cite{Drag75} established  the following remarkable inequality (\ref{eqeq14}) 
related to inner product 
of three vectors, and  presented some improvements of the celebrated Schwarz inequality 
in complex inner product spaces and provided  numerous application for $n$-tuples of complex numbers;  
see, e.g.,  \cite[p. 38]{Drag07} and \cite{Drag19} for more details. 
We here give an alternative proof 
from the perspective of matrix analysis
by using Corollary \ref{coro32}.

\begin{theorem} \label{thm32} 
Let $u,v$ and $w$ be vectors in an inner product space. Then 
\begin{equation}  \label{eqeq14} \begin{aligned}
& \left( \norm{u}^2 \norm{w}^2 -\bigl| \inner{u,w} \bigr|^2 \right) 
\left( \norm{w}^2 \norm{v}^2 -\bigl| \inner{w,v} \bigr|^2 \right) \\
&\quad \ge
\bigl( \left|\inner{u,w}\inner{w,v}\right| - 
\left|\inner{u,v}\inner{w,w}\right|  \bigr)^2. 
\end{aligned} \end{equation}
\end{theorem}

\begin{proof}
Without loss of generality, 
we may assume by scaling that 
$u,v$ and $w$ are unit vectors. 
The required inequality  can be  written  as 
\[  \left( 1-\bigl| \inner{u,w}\bigr|^2 \right) 
\left( 1-\bigl| \inner{w,v}\bigr|^2 \right) 
\ge 
 \left( \bigl| \inner{u,w}\bigr|  \bigl| \inner{w,v}\bigr| - 
\bigl| \inner{u,v}\bigr| \right)^2, \]
which is equivalent to prove 
\begin{equation}\label{eq3}
1 + 2\bigl| \inner{u,v}\bigr| \bigl|\inner{v,w}\bigr|  \bigl| \inner{w,u}\bigr| 
\ge 
\bigl| \inner{u,v} \bigr|^2+ \bigl| \inner{v,w}\bigr|^2 + \bigl| \inner{w,u}\bigr|^2. 
\end{equation}
By Corollary \ref{coro32}, it follows that 
\begin{equation*} \begin{bmatrix}
1 & \bigl| \inner{u,v}\bigr| & \bigl|\inner{u,w}\bigr| \\[0.15cm]
\bigl|\inner{v,u}\bigr| & 1 & \bigl|\inner{v,w}\bigr| \\[0.15cm]
\bigl|\inner{w,u}\bigr|  & \bigl|\inner{w,v}\bigr| & 1
\end{bmatrix} 
\end{equation*}
is positive semidefinite. Hence, the determinant 
of this matrix is nonnegative, which yields the required inequality 
(\ref{eq3}). 
\end{proof}

\noindent 
{\bf Remark.}~~
When the block matrix $[A_{ij}]_{i,j=1}^m$ is positive semidefinite,
we know in Section \ref{sec2} that the $m\times m$ matrix $[\det A_{ij}]_{i,j=1}^m$ 
is also positive semidefinite. Invoking this fact,   
Zhang \cite[Theorem 5.1]{Zhang12}
 proved the following interesting inequality: 
If $u,v$ and $w$ are  unit vectors in an inner product space, then 
\begin{equation} \label{eqZhang}
1+2\, \mathrm{Re}\left( \inner{u,v} \inner{v,w} \inner{w,u}\right) 
\ge \bigl| \inner{u,v}\bigr|^2 + \bigl| \inner{v,w}\bigr|^2+ \bigl| \inner{w,u}\bigr|^2. 
\end{equation} 
Inequality (\ref{eq3}) is weaker than (\ref{eqZhang}). 
We remark here that 
(\ref{eqZhang}) implies the following improvement 
of   Theorem \ref{thm32},    
\begin{equation*}  \begin{aligned}
& \left( \norm{u}^2 \norm{w}^2 -\bigl| \inner{u,w} \bigr|^2 \right) 
\left( \norm{w}^2 \norm{v}^2 -\bigl| \inner{w,v} \bigr|^2 \right) \\
&\quad \ge
\bigl| \inner{u,w}\inner{w,v} - 
\inner{u,v}\inner{w,w}  \bigr|^2. 
\end{aligned} \end{equation*} 
This inequality is also proved by Dragomir in  \cite{Drag75}. 
We leave the details for the interested reader.  
Theorem \ref{thm32} motivates the author to  consider the triangle inequalities 
in next Section \ref{sec4}.

\section{Some Triangle inequalities} 
\label{sec4}

Let $V$ be an inner product space with the inner product $\inner{\cdot ,\cdot }$ 
over the real number field $\mathbb{R}$ or the complex number field $\mathbb{C}$. 
For any two nonzero vectors $u,v$ in $V$, 
there are two different ways 
to define the angle between the vectors $u$ and $v$ 
in terms of the inner product. Namely,
\[ 
 \Phi (u,v) =\arccos \frac{\mathrm{Re} \inner{u,v}}{\norm{u}\norm{v}}, \] 
 and 
 \[ 
 \Psi (u,v) =\arccos \frac{\bigl| \inner{u,v} \bigr|}{\norm{u}\norm{v}}. 
\]
There are various reasons and advantages that 
the angles are defined in these ways. 
Both two definitions are frequently used in the literature; 
see, e.g., \cite{Lin12,Ota20,Zhang16a} for more recent results.

The angles $\Phi$ and $\Psi$ are closely related, 
but not equal unless $\inner{u,v}$ is a nonnegative number. 
We can  see that $0\le \Phi \le \pi$ and $0\le \Psi \le \pi/2$, and 
$\Phi (u,v) \ge \Psi (u,v)$ for all $u,v\in V$, 
since $\mathrm{Re} \inner{u,v}\le \left| \inner{u,v}\right|$ and 
$f(x)=\arccos x$ is a decreasing function on $ [-1,1]$. 
Moreover, 
it is easy to verify that 
\begin{equation} \label{eeq21}
\Psi (u,v)=\min\limits_{|p|=1} \Phi (pu,v)= \min\limits_{|q|=1} \Phi (u,qv) =
\min\limits_{|p|=|q|=1} \Phi (pu,qv). 
\end{equation}

There exist two well-known triangle inequalities for $\Phi$ and $\Psi$ in the literature, 
we will state it as the following Theorem~\ref{thm41}.

\begin{theorem} \label{thm41}
Let $u,v$ and $w$ be vectors in an inner product space. Then 
\begin{align}
\label{eeq22} \Phi (u,v) &\le \Phi (u,w) +\Phi (w,v),\\
\label{eeq23} \Psi (u,v) &\le \Psi (u,w) +\Psi (w,v).
\end{align}
\end{theorem}

The first  inequality (\ref{eeq22}) is attributed to Krein 
who presented briefly the result without proof in \cite{Krein69}, 
and proved first by Rao \cite{Rao76} and then in \cite[p. 56]{GR97}, 
whose proof boils down to the positivity of the matrix (\ref{eeq12}). 
We remark that the case on real field for (\ref{eeq22}) 
can  be seen in \cite[p. 31]{Zhang11}. 

For the second one, Lin \cite{Lin12} illustrated  that 
(\ref{eeq23}) can be deduced from (\ref{eeq22}) because of the relation (\ref{eeq21}). 
 Moreover, 
  the positivity of the matrix in Corollary \ref{coro32}
 also guarantees  the triangle inequality (\ref{eeq23}). 
 It is noteworthy that Theorem \ref{thm32} also 
 implies (\ref{eeq23}).  
Indeed, by (\ref{eqeq14}), we have 
\begin{align*}
& \left( \norm{u}^2 \norm{w}^2 -\bigl| \inner{u,w} \bigr|^2 \right)^{1/2} 
\left( \norm{w}^2 \norm{v}^2 -\bigl| \inner{w,v} \bigr|^2 \right)^{1/2} \\
&\quad \ge
\bigl| \inner{u,w}\inner{w,v}\bigr| -\bigl| \inner{u,v}\inner{w,w}  \bigr|. 
\end{align*}
By dividing with $\norm{u}\norm{v}\norm{w}^2$, we can obtain 
\[ \frac{\left| \inner{u,v}\right|}{\norm{u}\norm{v}} 
\ge \frac{\left| \inner{u,w}\right|}{\norm{u}\norm{w}} 
\frac{\left| \inner{w,v}\right|}{\norm{w}\norm{v}} -
\sqrt{1-\frac{\left| \inner{u,w}\right|}{\norm{u}\norm{w}}}
\sqrt{1-\frac{\left| \inner{w,v}\right|}{\norm{w}\norm{v}}}, \]
which is equivalent to 
\begin{align*} \cos \Psi (u,v) 
&\ge \cos \Psi (u,w) \cos \Psi (w,v) - 
\sin \Psi (u,w) \sin \Psi (w,v) \\
&=\cos (\Psi (u,w) +\Psi (w,v)). 
\end{align*}
Thus, (\ref{eeq23}) follows by the decreasing property of cosine on $[0,\pi]$. 

In this section, 
we present an intuitive proof of 
inequalities (\ref{eeq22}) and (\ref{eeq23}). 
 Our method  can be viewed as a further development 
  of that in \cite[p. 31]{Zhang11}, 
and allows us to provide some other angle inequalities. 

\begin{proof}[Proof of Theorem~\ref{thm41}]
We here only prove (\ref{eeq23}) since (\ref{eeq22}) can be proved in a  similar way. 
Because the desired inequality involves 
only three vectors $u,v$ and $w$, 
we may focus on the subspace spanned by $u,v$ and $w$, 
which has dimension at most $3$. 
We may further choose an orthonormal basis (a unit vector 
in the case of dimension one) of this subspace $\mathrm{Span}\{u,v,w\}$. 
Assume that $u,v $ and $w$ have coordinate vectors $x,y$ and $z$ under this basis, respectively. 
Then the desired inequality holds  if and only if it holds for complex vectors $x,y$ and $z$ with 
the standard inner product 
\[ \inner{x,y}=\overline{y_1}x_1+\overline{y_2}x_2 +\cdots +\overline{y}_nx_n.\] 
That is to say, our main goal is to show the following: 
\begin{equation}\label{eeq24} 
\Psi (x,y)\le \Psi (x,z) + \Psi (z,y),\quad \forall \, x,y,z \in \mathbb{C}^3.
\end{equation}
We next prove the  inequality (\ref{eeq24}) in two steps. 
Suppose first that the inner product space is a Euclidean space 
(i.e., an inner product space over field $\mathbb{R}$). 
Then the problem is reduced to $\mathbb{R},\mathbb{R}^2$ or $\mathbb{R}^3$ 
depending on whether the dimension of $\mathrm{Span}\{u,v,w\}$ is $1,2$ or $3$, respectively. 
In this real case, one can draw a simple graph to get the result. 
If the inner product space is an unitary space 
(i.e., an inner product space over field $\mathbb{C}$),  
we need to apply some technical tricks. 
We observe that 
the desired inequality (\ref{eeq24}) 
is not changed if we replace $x,y$ with 
$\omega x,\delta y$ for any complex numbers $\omega,\delta$ satisfying $|\omega|=|\delta|=1$. 
Therefore, we may assume further that both $\inner{x,z}$ and $\inner{z,y}$ are real numbers. 
Let $x=x_1+ix_2,y=y_1+iy_2$ and $z=z_1+iz_2$ for some vectors $x_i,y_i,z_i\in \mathbb{R}^3(i=1,2)$ 
and denote by 
\[ X=\begin{bmatrix}x_1 \\x_2 \end{bmatrix},\quad 
Y=\begin{bmatrix}y_1 \\y_2 \end{bmatrix},\quad 
Z=\begin{bmatrix}z_1 \\z_2 \end{bmatrix}. \]
Note that $X,Y,Z\in \mathbb{R}^6$, 
then by the previous statement for Euclidean space, we can get 
\begin{equation} \label{eeq25}
 \Psi (X,Y) \le \Psi (X,Z) + \Psi (Z,Y). 
\end{equation}
Since $\inner{x,z}$ and $\inner{z,y}$ are real numbers, we have 
\begin{align*}
\inner{x,z}&=\re \inner{x,z}=z_1^Tx_1 +z_2^Tx_2=\inner{X,Z}, \\
\inner{z,y}&=\re \inner{z,y}=y_1^Tz_1+y_2^Tz_2=\inner{Z,Y},\\
\inner{x,y}&=y_1^Tx_1 +y_2^Tx_2 +i(y_1^Tx_2 -y_2^Tx_1) .
\end{align*}
It is easy to see that $\norm{x}=\norm{X},\norm{y}=\norm{Y}$ and $\norm{z}=\norm{Z}$. Thus, 
\begin{equation}\label{eeq26} 
\Psi (x,z)=\Psi (X,Z),\quad \Psi (z,y)=\Psi (Z,Y). 
\end{equation}
Since $f(t)=\mathrm{arccos} \,(t)$ is a decreasing function on $[-1,1]$, we get 
\begin{equation}\label{eeq27}
 \Psi (x,y) =\arccos \frac{\bigl| \inner{x,y} \bigr|}{\norm{x}\norm{y}}\le 
\frac{\bigl| y_1^Tx_1 +y_2^Tx_2 \bigr|}{\norm{X}\norm{Y}}=\Psi (X,Y). 
\end{equation}
Combining  (\ref{eeq25}), (\ref{eeq26}) and (\ref{eeq27}), 
we can get the desired inequality (\ref{eeq24}).
\end{proof}

Using the same idea and technique of the proof of Theorem \ref{thm41}, 
one could also get the following angle inequalities. 

\begin{proposition} \label{prop42}
Let $u,v$ and $w$ be vectors in an inner product space. Then 
\begin{gather*}
\left| \Phi (u,v) -\Phi (v,w) \right|  \le 
\Phi (u,w) 
\le \Phi (u,v) +\Phi (v,w), \\
0\le \Phi (u,v) + \Phi (v,w) +\Phi (w,u) \le 2\pi.
\end{gather*}
Moreover, the above inequalities hold for $\Psi$. 
\end{proposition}

To end this paper, we give an application 
of Proposition \ref{prop42}. 
The following elegant inequality is the main result in \cite{Zhang94} 
and also can be found in \cite[p. 195]{Zhang11} and \cite{Lin12}, 
it was  derived as a tool in showing a trace inequality for complex unitary matrices. 
Of course, the line of proof provided here is different. 

\begin{corollary}
Let $u,v$ and $w$ be vectors in an inner product space over $\mathbb{C}$. Then 
\[ \sqrt{1-\frac{\left| \inner{u,v}\right|^2}{\norm{u}^2\norm{v}^2}} 
\le \sqrt{1-\frac{\left| \inner{u,w}\right|^2}{\norm{u}^2\norm{w}^2}}  +
\sqrt{1-\frac{\left| \inner{w,v}\right|^2}{\norm{w}^2\norm{v}^2}} . \]
Moreover, the inequality also holds if we replace $|\cdot |$ with $\re \,(\cdot)$. 
\end{corollary}

\begin{proof}
For brevity, we denote $\alpha,\beta,\gamma$ by the angles 
$\Psi (u,v),\Psi (u,w)$, $\Psi (w,v)$ or $\Phi (u,v),\Phi (u,w)$, $\Phi (w,v)$, respectively. 
By Proposition \ref{prop42}, we have 
\begin{gather*} 
0\le \frac{\alpha}{2} \le \frac{\beta +\gamma}{2} \le \pi -\frac{\alpha}{2},\quad 
0\le \frac{|\beta-\gamma|}{2} \le \frac{\alpha}{2} \le \frac{\pi}{2}.
\end{gather*}
Then 
\begin{align*}
0\le \sin \frac{\alpha}{2} \le \sin \frac{\beta +\gamma}{2},\quad 
0\le \cos \frac{\alpha}{2} \le \cos \frac{\beta-\gamma}{2}. 
\end{align*}
The required inequality can be written 
\[ \sin \alpha \le 
2\sin \frac{\beta +\gamma}{2} \cos \frac{\beta-\gamma}{2} 
=\sin \beta + \sin \gamma.\]
This completes the proof. 
\end{proof}

\section*{Acknowledgments}
All authors would like to express sincere thanks to 
Prof. Fuzhen Zhang 
for his kind help and valuable discussion \cite{BP10} before its publication, 
which considerably improves the presentation of our manuscript.
Finally, the second author is grateful for the valuable comments 
and  suggestions 
from Prof. Meiyue Shao. 
This work was supported by NSFC (Grant No. 11871479 and 12071484),  
Hunan Provincial Natural Science Foundation (Grant No. 2020JJ4675 and 2018JJ2479) 
and Mathematics and Interdisciplinary Sciences Project of CSU.

\end{document}